\documentclass[12pt]{amsart}

\usepackage{fullpage}
\usepackage{amsmath}
\usepackage{amssymb}
\usepackage{amsthm}
\usepackage{amstext}
\usepackage{appendix}
\usepackage{perpage}
\usepackage{slashed}
\usepackage[x11names]{xcolor}
\usepackage{tikz-cd}
\usepackage{ytableau}\ytableausetup{centertableaux}
\usepackage[colorlinks=true, allcolors=blue]{hyperref}
\usepackage{todonotes}

\newtheorem{thm}{Theorem}[section]

\newtheorem{coro}[thm]{Corollary}
\newtheorem{prop}[thm]{Proposition}

\theoremstyle{definition}

\newtheorem{defi}[thm]{Definition}
\newtheorem{remark}[thm]{Remark}

\begin{document}

\title{Class numbers and invariant characters of $\mathfrak{sl}_2(\mathbb{F}_p)$}

\author{Zhe Chen \and Yongqi Feng}

\address{Department of Mathematics, Shantou University, Shantou, 515821, China}
\email{zhechencz@gmail.com}

\address{}
\email{yfeng1@foxmail.com}

\begin{abstract}
Let $p$ be a prime and let $S_2(\Gamma(p))$ be the space of weight $2$ cusp forms for the principal congruence subgroup $\Gamma(p)$. Then $\mathrm{SL}_2(\mathbb{F}_p)$ acts on $S_2(\Gamma(p))$ in a natural way. Around 1928, Hecke proved that if  $p>3$ and $p\equiv 3\mod 4$, then the class number of $\mathbb{Q}(\sqrt{-p})$ is equal to the difference between the multiplicities of two particular irreducible representations of $\mathrm{SL}_2(\mathbb{F}_p)$  in  $S_2(\Gamma(p))$. In this paper we prove a Lie algebra analogue of this result. As an application we extend Hecke's result to $\mathrm{SL}_2(\mathbb{Z}/p^r)$ (acting on $S_2(\Gamma(p^r))$) for any $r\geq 2$. 
\end{abstract}

\maketitle

\section{Introduction}

Let $p$ be a prime and let $S_2(\Gamma(p))$ be the space of weight two cusp forms for the principal congruence subgroup $\Gamma(p)$. Since $\mathrm{SL}_2(\mathbb{F}_p)$ acts on the modular curve $X(p)$ of level $p$, it also acts on the space $S_2(\Gamma(p))$ in a natural way, by viewing the latter as the space of holomorphic differential forms on $X(p)$. 

\vspace{2mm} From now on, we assume that $p$ is odd.

\vspace{2mm} Then the finite group $\mathrm{SL}_2(\mathbb{F}_p)$ admits exactly two  (inequivalent) irreducible representations of dimension $\frac{p-1}{2}$; see e.g.\ \cite[Table~12.1]{DM_book_2nd_edition}. They are both cuspidal representations. Let us denote  by $m_{+}$ and $m_-$  their multiplicities  in  $S_2(\Gamma(p))$, respectively.  Hecke \cite{Hecke_1928_cuspform_MathZ,Hecke1928,Hecke_1931_modulfunktionen} showed that: If  $p>3$ and $p\equiv 3\mod 4$, then
$$m_+-m_-=h(-p),$$
where $h(-p)$ denotes the class number of $\mathbb{Q}(\sqrt{-p})$. Details and modern expositions  on Hecke's work can be found in \cite{Casselman_2017_Hecke, Gross_2018_Hecke}.

\vspace{2mm} In this paper we prove an analogue of Hecke's result for \emph{invariant characters} (see Section~\ref{sec: invariant characters}) of the Lie algebra $\mathfrak{sl}_2(\mathbb{F}_p)$. Instead of $S_2(\Gamma(p))$, we work with $S_2(\Gamma(p^r))$ for any $r\geq 2$. This starts with two observations: 
\begin{itemize}
\item The additive group of $\mathfrak{sl}_2(\mathbb{F}_p)$ acts on $S_2(\Gamma(p^r))$ via taking the kernel of reduction map $\mathfrak{sl}_2(\mathbb{F}_p)\cong \mathrm{Ker}(\mathrm{SL}_2(\mathbb{Z}/p^r)\rightarrow\mathrm{SL}_2(\mathbb{Z}/p^{r-1}))$;
\item  $\mathfrak{sl}_2(\mathbb{F}_p)$ admits exactly two irreducible invariant characters associated to the regular nilpotent orbits.
\end{itemize}
Denote  by $n_+$ and $n_-$ the multiplicities of the above  two invariant characters in $S_2(\Gamma(p^r))$, respectively. Our main result is Theorem~\ref{thm:main}, which  asserts that, if $p>3$ and $p\equiv 3\mod 4$, then 
$$n_+-n_-=p^{{2r-3}}\cdot h(-p).$$
The   proof roughly consists of three steps:  The first step is a ``normalisation'' of the  regular nilpotent orbit invariant characters of $\mathfrak{sl}_2(\mathbb{F}_p)$, which allows one to simplify the computation of $n_+-n_-$ using Gauss sums; the second step is to apply an idea used in  \cite{Casselman_2017_Hecke,Gross_2018_Hecke}, namely to use the Woods Hole fixed point formula to transfer the computation into  evaluation of  differentials; 
the third step is to use the structure of   $\mathrm{SL}_2(\mathbb{Z}/p^r)$ to compute the fixed points (and then the corresponding differentials) on the modular curve $X(p^r)$. 

\vspace{2mm} In Corollary~\ref{coro: extend Hecke} we   give an extension of Hecke's result to $\mathrm{SL}_2(\mathbb{Z}/p^r)$ for any $r\geq 2$; this is achieved by combining the above main result and Clifford's theorem. In this $r\geq2$ case one need to consider a family of pairs of dual representations, not just a single pair of dual representations (see  Remark~\ref{remark:single pair}). Over the past more than 90 years, there have been various extensions of Hecke's result, among which we mention the works of McQuillan \cite{McQuillan_1962_AJM}  and of Panda \cite{Panda_2020_JNT}: Instead of the case of prime powers $p^r$, they considered the case of products of distinct primes, so their results and ours may be viewed as complementary.

\vspace{2mm} In Proposition~\ref{prop:multi sum} we determine $n_++n_-$ by  studying the sum representation space $S_2(\Gamma(p^r))\oplus \overline{S_2(\Gamma(p^r))}$; note that this together with  Theorem~\ref{thm:main} allows one  to determine $n_+$ and $n_-$ separately.

\vspace{2mm} \noindent {\bf Acknowledgement.} This work is inspired by  \cite{Casselman_2017_Hecke, Gross_2018_Hecke}. During the preparation, ZC is partially supported by the  Natural Science Foundation of Guangdong No.~2023A1515010561, and YF by the  Natural Science Foundation of Guangdong No.~2022A1515010912.

\section{Invariant characters}\label{sec: invariant characters}

Let $G$  be a connected reductive group, defined over a complete discrete valuation ring $\mathcal{O}$ with a fixed uniformiser $\pi$ and finite residue field $\mathbb{F}_q$. Let $\mathfrak{g}$ be the Lie algebra of $G_{\mathbb{F}_q}$. Then for any $r\geq 2$, we have 
$$\mathfrak{g}(\mathbb{F}_q)\cong\mathrm{Ker}(G(\mathcal{O}/\pi^r)\longrightarrow G(\mathcal{O}/\pi^{r-1})),$$
in which the left hand side is viewed as the additive group and the right hand side is given by the reduction map modulo $\pi^{r-1}$.  In particular, the additive group of $\mathfrak{g}(\mathbb{F}_q)$ can be viewed as a normal subgroup of $G(\mathcal{O}/\pi^r)$, and hence $G(\mathcal{O}/\pi^r)$ acts on  $\mathfrak{g}(\mathbb{F}_q)$ through conjugation; this conjugation action factors through $G(\mathbb{F}_q)$ (viewed as a quotient group of $G(\mathcal{O}/\pi^r)$). 

\vspace{2mm} For more details on the above construction, see \cite[II]{Demazure_Gabriel_1980intro_AG_AG}. 

\vspace{2mm} In the remaining part of this paper we always assume that $r\geq2$.

\begin{defi}
Let $\chi$ be a complex (and not necessarily irreducible) character of the additive group $\mathfrak{g}(\mathbb{F}_q)$. If $\chi(gxg^{-1})=\chi(x)$ for any $g\in G(\mathcal{O}/\pi^r)$ and any $x\in \mathfrak{g}(\mathbb{F}_q)$, then we call $\chi$ an invariant character of $\mathfrak{g}(\mathbb{F}_q)$ with respect to $G(\mathcal{O}/\pi^r)$.
\end{defi}

For $r=2$ and $\mathrm{char}(\mathcal{O})>0$, invariant characters play an important role in the representation theory of $G(\mathbb{F}_q)$, and for general $r\geq2$ and $\mathcal{O}$ they are closely related with the smooth representations of $G(\mathcal{O})$; see e.g.\ \cite{Lusztig_1987_fourier_Lie_alg, Let2005book, Let09CharRedLieAlg, ChenStasinski_2016_algebraisation, ChenStasinski_2023_algebraisation_II}. 

\vspace{2mm} Below we focus on the case that $G={\mathrm{SL}_2}$ and $\mathcal{O}=\mathbb{Z}_p$.

\vspace{2mm} Let $\slashed{\square}$ be a fixed quadratic non-residue modulo $p$. Then, since $p$ is odd, there are precisely two regular nilpotent $\mathrm{SL}_2(\mathbb{Z}/p^r)$-orbits in $\mathfrak{sl}_2(\mathbb{F}_p)$, with representatives
\begin{equation*}
u:=\begin{bmatrix}
0 & 0  \\
1 & 0
\end{bmatrix}
\quad \textrm{and} \quad
v:=\begin{bmatrix}
0 & 0   \\
\slashed{\square} & 0
\end{bmatrix},
\end{equation*}
respectively (viewing $\mathfrak{sl}_2(\mathbb{F}_p)$ as a normal subgroup of $\mathrm{SL}_2(\mathbb{Z}/p^r)$, these two matrices are understood as 
$\begin{bmatrix}
1 & 0  \\
p^{r-1} & 1
\end{bmatrix}$
and
$\begin{bmatrix}
1 & 0  \\
p^{r-1}\slashed{\square} & 1
\end{bmatrix}$,
respectively).

\vspace{2mm} Let $\psi\colon \mathbb{F}_p\rightarrow \mathbb{C}^{\times}$ be the non-trivial irreducible character of the additive group $\mathbb{F}_p$ taking $1$ to $\zeta_p:=\exp\left(\frac{2\pi i}{p}\right)$. Then we get two invariant characters of $\mathfrak{sl}_2(\mathbb{F}_p)$:
\begin{equation*}
\chi_{+}(-):=\sum_{ h\in \mathrm{SL}_2(\mathbb{F}_p)/C_{\mathrm{SL}_2(\mathbb{F}_p)}(u) }\psi\left( \mathrm{Tr} (u^h\cdot(-)) \right)
\end{equation*}
and
\begin{equation*}
\chi_{-}(-):=\sum_{ h\in \mathrm{SL}_2(\mathbb{F}_p)/C_{\mathrm{SL}_2(\mathbb{F}_p)}(v) }\psi\left( \mathrm{Tr} (v^h\cdot(-)) \right).
\end{equation*}
Note that these two invariant characters are irreducible in the sense that they cannot be decomposed into a sum of two non-zero invariant characters (see \cite[Subsection~2.4]{Let09CharRedLieAlg}).

\section{The main result}\label{sec:main result}

Since $\mathrm{SL}_2(\mathbb{Z}/p^r)$ acts on $S_2(\Gamma(p^r))$, so does $\mathfrak{sl}_2(\mathbb{F}_p)$ (viewed as a normal subgroup of $\mathrm{SL}_2(\mathbb{Z}/p^r)$). By Clifford's theorem (see e.g.\ \cite[6.2]{Isaacs_CharThy_Book}) we know that $S_2(\Gamma(p^r))$ decomposed into a direct sum of irreducible invariant representations of $\mathfrak{sl}_2(\mathbb{F}_p)$. 

\vspace{2mm} Let us denote by $n_{+}$ (resp.\ $n_{-}$) the multiplicity of $\chi_{+}$ (resp.\ $\chi_-$) in $S_2(\Gamma(p^r))$.

\begin{thm}\label{thm:main}
If $p>3$ and $p\equiv 3\mod 4$, then 
$$n_{+}-n_{-}=p^{2r-3}\cdot h(-p).$$
If $p\equiv 1 \mod 4$, then $n_+-n_-=0$. If $p=3$, then $n_{+}-n_{-}=3^{2r-4}=p^{2r-4}\cdot h(-p)$.
\end{thm}

\begin{proof}
Consider the following characters of $\mathfrak{sl}_2(\mathbb{F}_p)$
$$N_+:=\frac{1}{2} \cdot \sum_{s\in T} \psi\left( \mathrm{Tr} ( {u}^s\cdot (-) ) \right)$$
and
$$N_-:=\frac{1}{2} \cdot \sum_{s\in T} \psi\left( \mathrm{Tr} ( {v}^s\cdot (-) ) \right),$$
in which $T$ denotes the group of diagonal matrices in $\mathrm{SL}_2(\mathbb{F}_p)$. By construction both $N_+$ and $N_-$ are multiplicity-free and are of degree $\frac{p-1}{2}$; moreover the multiplicity of every irreducible constituent of $N_+$ (resp.\ $N_-$) in $S_2(\Gamma(p^r))$ is $n_+$ (resp.\ $n_-$). 

\vspace{2mm} Let us write $\chi_r$ for the character of  $S_2(\Gamma(p^r))$ (viewed as a representation space of $\mathfrak{sl}_2(\mathbb{F}_p)$). Then we have
\begin{equation*}
\begin{split}
\langle  N_+-N_-, \chi_r \rangle_{\mathfrak{sl}_2(\mathbb{F}_p)}
&=\langle  N_+-N_-,n_+N_+ + n_-N_-   \rangle_{\mathfrak{sl}_2(\mathbb{F}_p)}\\
&=n_+\langle  N_+,N_+  \rangle_{\mathfrak{sl}_2(\mathbb{F}_p)}
- n_-\langle  N_-,N_-  \rangle_{\mathfrak{sl}_2(\mathbb{F}_p)}\\
&=\frac{p-1}{2}\cdot (n_{+}-n_{-}),
\end{split}
\end{equation*}
in which the inner product is with respect to the space of characters of the additive group of $\mathfrak{sl}_2(\mathbb{F}_p)$. From this we get
\begin{equation}\label{temp:1}
\begin{split}
n_{+}-n_{-}
&=\frac{2}{p-1}\cdot\langle  N_+-N_-, \chi_r  \rangle_{\mathfrak{sl}_2(\mathbb{F}_p)}\\
&=\frac{2}{(p-1)p^3}\cdot \sum_{g\in \mathfrak{sl}_2(\mathbb{F}_p)}  (N_+-N_-)(g)\cdot \chi_r(g^{-1}).
\end{split}
\end{equation}
Meanwhile, for a given $g\in M_2(\mathbb{F}_p)$ let us write 
$g=\begin{bmatrix}
a(g) & b(g)  \\
c(g) & d(g)
\end{bmatrix}$; we also need to use the quadratic Gauss sum $G\left( \frac{x}{p} \right)$ (see \cite[(3.27)]{Iwaniec_et_Kowalski_AnalyNumThyBk_2004}).  Then
\begin{equation}\label{temp2}
\begin{split}
(N_+-N_-)(g)
&=\frac{1}{2}\cdot 
\left(
\sum_{h\in \mathbb{F}_p} \zeta_p^{b(g)\cdot h^2} 
- \sum_{h'\in \mathbb{F}_p} \zeta_p^{b(g)\cdot \slashed{\square}\cdot  h'^2}
\right)\\
&=\frac{1}{2}\cdot\left(  G\left( \frac{b(g)}{p} \right)  - G\left( \frac{b(g)\cdot\slashed{\square}}{p} \right)   \right)\\
&=\left( \frac{b(g)}{p} \right) \cdot \sqrt{ (-1)^{\frac{p-1}{2}}\cdot p},
\end{split}
\end{equation}
in which the last equality follows from \cite[(3.29) and (3.22)]{Iwaniec_et_Kowalski_AnalyNumThyBk_2004}  (in the last row $\left( \frac{b(g)}{p} \right)$ means the Legendre symbol).

\vspace{2mm} Combining \eqref{temp:1} and \eqref{temp2} we see
\begin{equation}\label{temp3}
n_{+}-n_{-}
=\frac{2\cdot \sqrt{ (-1)^{\frac{p-1}{2}}\cdot p}}{(p-1)\cdot p^3} 
\cdot 
\sum_{g\in \mathfrak{sl}_2(\mathbb{F}_p)\setminus \{0\} }  \left( \frac{b(g)}{p} \right) \cdot \chi_r(g^{-1}).
\end{equation}
To continue we need to evaluate $\chi_r(g^{-1})$; first we follow \cite{Casselman_2017_Hecke, Gross_2018_Hecke} to apply the Woods Hole fixed point formula (or, the Atiyah--Bott--Eichler formula; see \cite[(4.14)]{Atiyah_et_Bott_Lefschetz_II_1968}):
$$\chi_r(g^{-1})=1-\sum_{x\in X(p^r);\ gx=x}\frac{1}{1-dg_x},$$
where $dg_x$ denotes the differential of $g$ at $x$ and $X(p^r)$ denotes the modular curve for the principal congruence subgroup
$$\Gamma(p^r):=\mathrm{Ker}(\mathrm{SL}_2(\mathbb{Z})\longrightarrow\mathrm{SL}_2(\mathbb{Z}/p^r)).$$
So  \eqref{temp3} becomes
\begin{equation}\label{temp4}
n_{+}-n_{-}
=\frac{-2\cdot \sqrt{ (-1)^{\frac{p-1}{2}}\cdot p}}{(p-1)\cdot p^3} 
\cdot 
\sum_{g\in \mathfrak{sl}_2(\mathbb{F}_p) \setminus \{0\}}  \left( \frac{b(g)}{p} \right) \cdot \sum_{x\in X(p^r);\ gx=x}\frac{1}{1-dg_x}.
\end{equation}
There are two things need to be done: to compute the fixed points on $X(p^r)$ and to compute the differentials $dg_x$.

\vspace{2mm} Consider the covering map of modular curves
$$\pi\colon X(p^r)\longrightarrow X(1)$$
whose Galois group is $\mathrm{SL}_2(\mathbb{Z}/p^r)/\{ \pm I_2 \}$; recall that $X(p^r):=\Gamma(p^r)\backslash\mathbb{H}^*$ and $X(1):=\mathrm{SL}_2(\mathbb{Z})\backslash\mathbb{H}^*$, where $\mathbb{H}^*$ denotes the extended upper half-plane $\mathbb{H}^*:=\mathbb{H}\cup \mathbb{Q} \cup \{\infty\}$. If $x\in X(p^r)$ is fixed by some non-trivial $g\in \mathfrak{sl}_2(\mathbb{F}_p)$, then $\pi$ is ramified over $\pi(x)\in X(1)$, and hence $\pi(x)\in\{ i, e^{2\pi i/3}, \infty \}$;  the  stabiliser of those $x$ (in $\mathrm{SL}_2(\mathbb{Z}/p^r)$) would be correspondingly conjugated to one of the following subgroups
\begin{equation}\label{inertia groups}
\widetilde{G}_i:=\left\langle 
\begin{bmatrix}
 & 1 \\
-1 & 
\end{bmatrix} \right\rangle,
\quad
\widetilde{G}_{e^{2\pi i/3}}:=\left\langle 
\begin{bmatrix}
    & 1 \\
-1 & -1
\end{bmatrix}, -I_2 \right\rangle,
\quad
\widetilde{G}_{\infty}:=\left\langle
\begin{bmatrix}
1& 1 \\
 & 1
\end{bmatrix},\ -I_2 \right\rangle.
\end{equation}
For details, see \cite[Page~22]{Shimura_intro_1971}. 

\vspace{2mm} By direct computation one sees that 
$${^s\widetilde{G}_i\cap \mathfrak{sl}_2(\mathbb{F}_p)}={^s\widetilde{G}_i\cap \mathfrak{sl}_2(\mathbb{F}_p)}=I_2$$
for any $s\in \mathrm{SL}_2(\mathbb{Z}/p^r)$. (Viewing $\mathfrak{sl}_2(\mathbb{F}_p)$   as the kernel of the reduction map,  the last $I_2$ is understood as the zero matrix in $\mathfrak{sl}_2(\mathbb{F}_p)$.) In particular, if $g\in \mathfrak{sl}_2(\mathbb{F}_p)\setminus\{0\}$ fixes some $x\in X(p^r)$, then we must have that $x\in\pi^{-1}(\infty)$, in which case there is an $s\in \mathrm{SL}_2(\mathbb{Z}/p^r)$ satisfying that
\begin{equation*}
\begin{split}
C_{\mathfrak{sl}_2(\mathbb{F}_p)}(x)
&={^s\widetilde{G}_{\infty}\cap \mathfrak{sl}_2(\mathbb{F}_p)}\\
&={^s\left(\widetilde{G}_{\infty}\cap \mathfrak{sl}_2(\mathbb{F}_p)\right)}\\
&=
\left\{
I_2+p^{r-1}\cdot
s\begin{bmatrix}
0& m \\
 & 0
\end{bmatrix}s^{-1}
\right\}_{m=0,1,...,p-1}.
\end{split}
\end{equation*}
The elements in $\mathfrak{sl}_2(\mathbb{F}_p)$ that can be conjugated into 
$\left\{ \begin{bmatrix}
0 & m \\
  & 0
\end{bmatrix} \right\}_{m=1,...,p-1}$ form two conjugacy classes, with representatives $u_0:=\begin{bmatrix}
0 & 1 \\
  & 0
\end{bmatrix}$ and $v_0:=\begin{bmatrix}
0 & \slashed{\square} \\
  & 0
\end{bmatrix}$, respectively, and both classes have the cardinal 
$$\frac{\#\mathrm{SL}_2(\mathbb{Z}/p^r)}{\# C_{\mathrm{SL}_2(\mathbb{Z}/p^r)} (u_0) }=\frac{p(p-1)(p+1)p^{3(r-1)}}{2p\cdot p^{3(r-1)}}=\frac{(p-1)(p+1)}{2}.$$
Moreover, direct matrix computation tells that for $h\in \mathrm{SL}_2(\mathbb{Z}/p^r)$ one has either $\left(\frac{b(u_0^h)}{p}\right)=\left(\frac{b(u_0)}{p}\right)$ or $\left(\frac{b(u_0^h)}{p}\right)=0$; the latter case happens only when $u_0^h$ is a lower triangular matrix (and there are totally $\frac{p-1}{2}$ such possible lower triangular matrices). Combining these computations with \eqref{temp4} we get
\begin{equation}\label{temp5}
\begin{split}
n_+-n_-
=\frac{-2\cdot \sqrt{ (-1)^{\frac{p-1}{2}}\cdot p}}{(p-1)\cdot p^3} 
&\times 
\left(\frac{(p-1)(p+1)}{2}-\frac{p-1}{2}\right)\\
 &\times
\left(
\sum_{\substack{u_0x=x;\\ x\in\pi^{-1}(\infty)}}\frac{1}{1-d(u_0)_x} 
-\sum_{\substack{v_0x=x;\\ x\in\pi^{-1}(\infty)}}\frac{1}{1-d(v_0)_x}
\right).
\end{split}
\end{equation}
We shall simplify this sum by a further step. Consider $\frac{1}{p^r}\in\mathbb{Q}\subseteq\mathbb{H}^*$, and let $x_0$ be the image of $\frac{1}{p^r}$ in $X(p^r)=\Gamma(p^r)\backslash\mathbb{H}^*$. Then $x_0$ is equivalent to $\infty$ in $X(p^r)$ (see \cite[Lemma~3.8.2]{Diamond_et_Shurman_book2005}). Thus $u_0x_0=v_0x_0=x_0$ and $C_{\mathrm{SL}_2(\mathbb{Z}/p^r)}(x_0)=\widetilde{G}_{\infty}$. So, for  $h\in \mathrm{SL}_2(\mathbb{Z}/p^r)$ we have:
$$u(x_0^h)=x_0^h \iff u\in h^{-1}\widetilde{G}_{\infty}h \iff hu_0h^{-1}\in \widetilde{G}_{\infty} $$
and the last condition is equivalent to that $h$ modulo $p$ is upper triangular; denote the set consisting of those $h\in \mathrm{SL}_2(\mathbb{Z}/p^r)$ by $P$.  Then
\begin{equation}\label{temp6}
\begin{split}
\sum_{\substack{u_0x=x;\\ x\in\pi^{-1}(\infty)}}\frac{1}{1-d(u_0)_x} 
&=\sum_{h\in P/\widetilde{G}_{\infty}}\frac{1}{1-d(u_0)_{x_0^h}} \\
&=\sum_{h\in P/\widetilde{G}_{\infty}}\frac{1}{1-d(u_0^h)_{x_0}} \\
&=\frac{p^{2(r-1)}}{2}\sum_{\substack{h=\mathrm{diag}(a,a^{-1});\\ a\in\mathbb{F}_p\setminus\{0\} }}\frac{1}{1-d(u_0^h)_{x_0}} \\
&=\frac{p^{2(r-1)}}{2}\sum_{u'\in \mathfrak{R} }\frac{1}{1-du'_{x_0}} 
\end{split}
\end{equation}
in which $\mathfrak{R}:=\left\{ \begin{bmatrix}
1 & p^{r-1}a^2 \\
  & 1
\end{bmatrix} \right\}_{a\in \mathbb{F}_p\setminus\{0\}}\subseteq \mathfrak{sl}_2(\mathbb{F}_p)$ (viewed as a subset of $\mathrm{SL}_2(\mathbb{Z}/p^r)$). Similarly we have
\begin{equation}\label{temp7}
\sum_{\substack{v_0x=x;\\ x\in\pi^{-1}(\infty)}}\frac{1}{1-d(v_0)_x} 
=\frac{p^{2(r-1)}}{2}\sum_{v'\in \mathfrak{N} }\frac{1}{1-dv'_{x_0}} 
\end{equation}
in which $\mathfrak{N}:=\left\{ \begin{bmatrix}
1 & p^{r-1}a^2\slashed{\square} \\
  & 1
\end{bmatrix} \right\}_{a\in \mathbb{F}_p\setminus\{0\}}\subseteq \mathfrak{sl}_2(\mathbb{F}_p)$ (viewed as a subset of $\mathrm{SL}_2(\mathbb{Z}/p^r)$).

\vspace{2mm} Combining \eqref{temp5}, \eqref{temp6}, and \eqref{temp7} we get
\begin{equation}\label{temp8}
n_+-n_-
=\frac{- p^{2(r-2)}\cdot \sqrt{ (-1)^{\frac{p-1}{2}}\cdot p}}{2} 
\cdot
\left(
\sum_{u'\in \mathfrak{R} }\frac{1}{1-du'_{x_0}} 
-\sum_{v'\in \mathfrak{N} }\frac{1}{1-dv'_{x_0}} 
\right).
\end{equation}
In the below we follow \cite[Page~537]{Shimura_1973_JMSJ} to compute the differentials.

\vspace{2mm} Take $\beta=\begin{bmatrix}
1 &   \\
-p^r & 1
\end{bmatrix}\in\mathrm{SL}_2(\mathbb{Z}/p^r)$; since $\beta\cdot x_0=\infty$,
$$t(z)=\exp(2\pi i\beta(z)/p^r)$$
is a local coordinate system near $x_0=\frac{1}{p^r}$. Meanwhile, note that $u'^{\beta}\cdot u'^{-1}\in \Gamma(p^r)$ for any $u'=\begin{bmatrix}
1 & p^{r-1}a^2 \\
  & 1
\end{bmatrix}\in \mathfrak{R}$, 
so $u'^{\beta}$ represents the operator $u'$ on $X(p^r)$. Thus the action of $u'\in \mathfrak{R}$ near $x_0$, in terms of the above local coordinate $t(z)$, becomes
\begin{equation*}
\begin{split}
t(u'^{\beta}z)
&=\exp(2\pi i \beta (u'^{\beta}z)/p^r)=\exp(2\pi i u'(\beta(z))/p^r)\\
&=\exp(2\pi i ({\beta}(z)+p^{r-1}a^2)/p^r)=  t(z)\cdot\zeta_p^{a^2},
\end{split}
\end{equation*}
which means that $du'_{x_0}=\zeta_p^{a^2}$. Therefore
\begin{equation}\label{temp1-1}
\sum_{u'\in \mathfrak{R} }\frac{1}{1-du'_{x_0}}
 = \sum_{a\in \mathbb{F}_p\setminus\{0\} }\frac{1}{1-\zeta_p^{a^2}}
\end{equation}
and similar computation gives that
\begin{equation}\label{temp1-2}
\sum_{v'\in \mathfrak{N} }\frac{1}{1-dv'_{x_0}}
 = \sum_{a\in \mathbb{F}_p\setminus\{0\} }\frac{1}{1-\zeta_p^{\slashed{\square}a^2}}.
\end{equation}
Furthermore, for any $b=1,...,p-1$,  we have the following useful identity in \cite[Page~9]{Gross_2018_Hecke}
\begin{equation}\label{temp-trick}
\frac{1}{1-\zeta_p^b}=-\frac{1}{p} \cdot \left(\zeta_p^b+2\zeta_p^{2b}+...+(p-1)\zeta_p^{(p-1)b}\right),
\end{equation} 
from which one sees easily that
\begin{equation}\label{temp1-3}
\sum_{a\in \mathbb{F}_p\setminus\{0\}}\frac{1}{1-\zeta_p^{a}}=\frac{p-1}{2}.
\end{equation}
Putting \eqref{temp1-1}, \eqref{temp1-2}, \eqref{temp1-3} into \eqref{temp8} we get
\begin{equation}\label{temp9}
n_+-n_-
=\frac{- p^{2(r-2)}\cdot \sqrt{ (-1)^{\frac{p-1}{2}}\cdot p}}{2} 
\cdot
\left(
-p+1+2\cdot \sum_{a\in \mathbb{F}_p\setminus\{0\} }\frac{1}{1-\zeta_p^{a^2}} 
\right),
\end{equation}
in which $\sum_{a\in \mathbb{F}_p\setminus\{0\} }\frac{1}{1-\zeta_p^{a^2}}$ is related to the quadratic Gauss sum.

\vspace{2mm} If $p>3$ and $p\equiv 3\mod 4$, then by \cite[Proposition~5]{Gross_2018_Hecke} we know
$$\sum_{a\in \mathbb{F}_p\setminus\{0\} }\frac{1}{1-\zeta_p^{a^2}}=\frac{p-1}{2}+h(-p)\sqrt{-p},$$
so \eqref{temp9} gives 
$$n_+-n_-=p^{2r-3}\cdot h(-p),$$
as desired. 

\vspace{2mm} If $p\equiv 1\mod 4$, then applying \eqref{temp-trick} to $b=a^2$ we get
$$\sum_{a\in \mathbb{F}_p\setminus\{0\} }\frac{1}{1-\zeta_p^{a^2}} =-\sum_{k=1,...,p-1}\frac{k}{p}\sum_{a=1,...,p-1}\zeta_p^{ka^2}
=\frac{p-1}{2}-\frac{\sqrt{p}}{p}
\cdot\sum_{k=1,...,p-1}k\left( \frac{k}{p} \right) $$
(the last equality follows from properties of Gauss sums; see \cite[(3.29)]{Iwaniec_et_Kowalski_AnalyNumThyBk_2004}). Note that since $p\equiv 1\mod 4$, we have
$$\Sigma:=\sum_{k=1,...,p-1}k\left( \frac{k}{p} \right)
=\sum_{k=1,...,p-1}k\left( \frac{p-k}{p} \right)
=\sum_{k=1,...,p-1}(p-k)\left( \frac{k}{p} \right),$$
so 
$$\Sigma+\Sigma=p\cdot \sum_{k=1,...,p-1}\left( \frac{k}{p} \right)=0.$$
Thus $\Sigma=0$, and hence \eqref{temp9} implies that $n_+-n_-=0$ in this case. 

\vspace{2mm} Finally, if $p=3$, then a direct computation with \eqref{temp9} gives that $n_+-n_-=3^{2r-4}$.
\end{proof}

\section{Groups over finite rings}\label{sec:group level}

Since $p$ is odd and $r\geq2$, every irreducible representation of $\mathrm{SL}_2(\mathbb{Z}/p^r)$ is assigned an orbit in $\mathfrak{sl}_2(\mathbb{F}_p)$ via the trace form (see \cite[Section~2]{ChenStasinski_2023_algebraisation_II}). Let us denote by $\mathcal{U}$ (resp.\ $\mathcal{V}$) the set of equivalent classes of irreducible representations of $\mathrm{SL}_2(\mathbb{Z}/p^r)$ corresponding to the orbit containing $u=\begin{bmatrix}
0 & 0   \\
1 & 0
\end{bmatrix}$ (resp.\ $v=\begin{bmatrix}
0 & 0   \\
\slashed{\square} & 0
\end{bmatrix}$). In particular, $\mathcal{U}\sqcup\mathcal{V}$ consists of all regular nilpotent orbit representations (in the sense of \cite[Definition~2.3]{ChenStasinski_2023_algebraisation_II}) of $\mathrm{SL}_2(\mathbb{Z}/p^r)$. 

\vspace{2mm} For each $
\sigma\in \mathcal{U}\cup\mathcal{V}$, let $m_{\sigma}$ be the multiplicity of $\sigma$ in $S_2(\Gamma(p^r))$. Hecke's result can be extended to general $r\geq2$ in the following form:

\begin{coro}\label{coro: extend Hecke}
Suppose that $p>3$ and $p\equiv 3\mod 4$. Then
$$\sum_{x\in \mathcal{U}} m_x
-\sum_{y\in \mathcal{V}} m_y
=p^{r-1}\cdot h(-p).$$
\end{coro}
\begin{proof}
From the proof of Theorem~\ref{thm:main} we see that $\chi_{\pm}(1)=\frac{(p-1)(p+1)}{2}$; meanwhile, each regular nilpotent orbit representation of $\mathrm{SL}_2(\mathbb{Z}/p^r)$ has dimension $\frac{p^{r-2}(p-1)(p+1)}{2}$ (see \cite[4.3]{Shalika_SL_2_2004}). So by Clifford's theorem 
$$\sum_{x\in \mathcal{U}} m_x\cdot p^{r-2}=n_{+}
\qquad\textrm{and}\qquad
\sum_{y\in \mathcal{V}} m_y\cdot p^{r-2}=n_{-}.$$
Now the assertion follows from Theorem~\ref{thm:main}.
\end{proof}

\begin{remark}\label{remark:single pair}
When $p\equiv 3\mod 4$, taking dual ${\sigma}\mapsto\bar{\sigma}$ gives a bijection between $\mathcal{U}$ and $\mathcal{V}$, as in this case one has $\overline{\chi_+}=\chi_-$. Then a natural question is whether the single pair version of Corollary~\ref{coro: extend Hecke} holds true, namely, should we have
$$m_{\sigma}-m_{\bar{\sigma}}=\frac{p^{r-1}\cdot h(-p)}{|\mathcal{U}|}$$ 
for every $\sigma\in\mathcal{U}$? The answer is negative: According to \cite{Shalika_SL_2_2004} one has $|\mathcal{U}|=2p^{r-1}$, but it is well-known that $\frac{h(-p)}{2}$ cannot be an integer for $p\equiv 3\mod 4$ (this can also be seen by the parity assertion in the below Proposition~\ref{prop:multi sum} combined with Theorem~\ref{thm:main}).
\end{remark}

\section{The sum space}\label{sec:sum space}

Let $S:=S_2(\Gamma(p^r))\oplus \overline{S_2(\Gamma(p^r))}$ where $\overline{S_2(\Gamma(p^r))}$ is understood as the dual representation of $S_2(\Gamma(p^r))$. This sum space allows us to determine $n_++n_-$:

\begin{prop}\label{prop:multi sum}
We have
$$n_++n_-= \frac{p^{2(r-2)}(p-1)(p^r+p^{r-1}-6)}{12}.$$
In particular, $n_++n_-$ is odd if and only if $p\equiv 3\mod 4$.
\end{prop}

\begin{proof}
First, let $G$ be $\mathrm{SL}_2(\mathbb{Z}/p^r)/\{ \pm I_2 \}$, and for $j\in\{i, e^{2\pi i/3}, \infty \}$  let $G_j=\widetilde{G}_j/\{ \pm I_2 \}$ (these $\widetilde{G}_j$'s are defined in \eqref{inertia groups}). Then we have (as representations of $G$)
\begin{equation}\label{temp-d}
S\cong \mathbb{C}[G]-\mathbb{C}[G/G_i]-\mathbb{C}[G/G_{e^{2\pi i/3}}]-\mathbb{C}[G/G_{\infty}]+2\cdot 1_G;
\end{equation}
see \cite[Page~31]{Weinstein_2007_thesis}. So it suffices to decompose $\mathrm{Res}_{\mathfrak{sl}_2(\mathbb{F}_p)}\mathrm{Ind}_{G_j}^{G}1$ 
for $j\in \{i, e^{2\pi i/3}, \infty\}$.
By Mackey's intertwining formula we have
\begin{equation}\label{temp-d1}
\begin{split}
\sum_{j\in\{i, e^{2\pi i/3}, \infty \}  }
\mathrm{Res}_{\mathfrak{sl}_2(\mathbb{F}_p)}\mathrm{Ind}_{G_j}^{G}1
=&\sum_{j \in\{i, e^{2\pi i/3}, \infty \} }
\mathrm{Res}_{\mathfrak{sl}_2(\mathbb{F}_p)}\mathrm{Ind}_{\widetilde{G}_j}^{\mathrm{SL}_2(\mathbb{Z}/p^r)}1\\
=&\sum_{ j\in\{i, e^{2\pi i/3}, \infty \} }
\sum_{s\in \mathfrak{sl}_2(\mathbb{F}_p)\backslash \mathrm{SL}_2(\mathbb{Z}/p^r)/\widetilde{G}_j }
\mathrm{Ind}^{  \mathfrak{sl}_2(\mathbb{F}_p) }_{  \mathfrak{sl}_2(\mathbb{F}_p) \cap {^s\widetilde{G}_j}  } 1 \\
=&
\sum_{ j=i,e^{2\pi i/3} } 
\# \left(\mathfrak{sl}_2(\mathbb{F}_p)\backslash \mathrm{SL}_2(\mathbb{Z}/p^r)/\widetilde{G}_j   \right)
\mathbb{C}[\mathfrak{sl}_2(\mathbb{F}_p)]\\
&+\sum_{s\in \mathfrak{sl}_2(\mathbb{F}_p)\backslash \mathrm{SL}_2(\mathbb{Z}/p^r)/\widetilde{G}_{\infty} }
\mathrm{Ind}^{  \mathfrak{sl}_2(\mathbb{F}_p) }_{  \mathfrak{sl}_2(\mathbb{F}_p) \cap {^s\widetilde{G}_{\infty}}  } 1\\
= &\ \frac{5(p^2-1)p^{3r-5}}{12} 
\mathbb{C}[\mathfrak{sl}_2(\mathbb{F}_p)]\\
&+
\sum_{s\in \mathfrak{sl}_2(\mathbb{F}_p)\backslash \mathrm{SL}_2(\mathbb{Z}/p^r)/\widetilde{G}_{\infty} }
\mathrm{Ind}^{  \mathfrak{sl}_2(\mathbb{F}_p) }_{  \mathfrak{sl}_2(\mathbb{F}_p) \cap {^s\widetilde{G}_{\infty}}  } 1,
\end{split}
\end{equation}
where the  last row follows from that, when $j\neq\infty$, one has $\mathfrak{sl}_2(\mathbb{F}_p)\cap \widetilde{G}_j=1$, and hence
$$\# \left(\mathfrak{sl}_2(\mathbb{F}_p)\backslash \mathrm{SL}_2(\mathbb{Z}/p^r)/\widetilde{G}_j   \right)
=\frac{|\mathrm{SL}_2(\mathbb{Z}/p^r)|}{|\mathfrak{sl}_2(\mathbb{F}_p)|\cdot|\widetilde{G}_j|}=\frac{(p^2-1)p^{3r-5}}{|\widetilde{G}_j|}.$$
Note that 
$$ \mathfrak{sl}_2(\mathbb{F}_p)\backslash \mathrm{SL}_2(\mathbb{Z}/p^r)/\widetilde{G}_{\infty} 
=\mathrm{SL}_2(\mathbb{Z}/p^{r-1}) / (\widetilde{G}_{\infty}/\mathfrak{sl}_2(\mathbb{F}_p)\cap \widetilde{G}_{\infty})$$  
(which has cardinal $\frac{(p^2-1)p^{2(r-2)}}{2}$), and that
$$\mathrm{Ind}^{  \mathfrak{sl}_2(\mathbb{F}_p) }_{  \mathfrak{sl}_2(\mathbb{F}_p) \cap {^s\widetilde{G}_{\infty}}  } 1
\cong \widetilde{\mathrm{Reg}_{{^s\mathfrak{b}}}},$$
where $\mathfrak{b}\subseteq\mathfrak{sl}_2(\mathbb{F}_p)$ consists of 
the lower triangular matrices and here the symbol $\widetilde{(-)}$ denotes the trivial extension from ${^s\mathfrak{b}}$ to $\mathfrak{sl}_2(\mathbb{F}_p)$ (in particular, the character of  
$\widetilde{\mathrm{Reg}_{{^s\mathfrak{b}}}}$ takes the value $p^2$ at the elements of the form $s \begin{bmatrix}
0 & *   \\
0 & 0
\end{bmatrix} s^{-1}$, and takes the value $0$ at other elements). Therefore by \eqref{temp-d} and \eqref{temp-d1} we get (as representations of $\mathfrak{sl}_2(\mathbb{F}_p)$)
\begin{equation}\label{temp-d2}
\begin{split}
S
&\cong 
\frac{(p^2-1)p^{3r-5}}{12} 
\mathbb{C}[\mathfrak{sl}_2(\mathbb{F}_p)]+2\cdot 1_{\mathfrak{sl}_2(\mathbb{F}_p)}
-\sum_{  \mathrm{SL}_2(\mathbb{Z}/p^{r-1}) / (\widetilde{G}_{\infty}/\mathfrak{sl}_2(\mathbb{F}_p)\cap \widetilde{G}_{\infty})  } 
\widetilde{\mathrm{Reg}_{{^s\mathfrak{b}}}}\\
&\cong
\frac{(p^2-1)p^{3r-5}}{12} 
\mathbb{C}[\mathfrak{sl}_2(\mathbb{F}_p)]+2\cdot 1_{\mathfrak{sl}_2(\mathbb{F}_p)}
-\frac{(p-1)p^{2(r-2)}}{2}\cdot \sum_{  \mathrm{SL}_2(\mathbb{Z}/p^{r-1}) / P  } 
\widetilde{\mathrm{Reg}_{{^s\mathfrak{b}}}},
\end{split}
\end{equation}
in which $P$ denotes the set of matrices in $ \mathrm{SL}_2(\mathbb{Z}/p^{r-1}) $ whose reduction modulo $p$ are upper-triangular. (Note that by \cite[1.6.4]{Shimura_intro_1971} we have 
$\dim S=2\dim S_2(\Gamma(p^r))=2+p^{2r}(p^r-6)(1-p^{-2})/12$, and one can verify by  direct computations that the right hand side of \eqref{temp-d2} also has this dimension.)

\vspace{2mm}
The two invariant characters $\chi_{+}$ and $\chi_-$ are either dual to each other (when $p\equiv 3\mod 4$) or both self-dual (when $p\equiv 1\mod 4$). But by \eqref{temp-d2} the multiplicities of $\chi_{+}$ and ${\chi_-}$ in $S$ are the same, so in both cases we have (again by \eqref{temp-d2})
$$n_++n_-
=\frac{(p^2-1)p^{3r-5}}{12} 
- \frac{(p-1)p^{2(r-2)}}{2}
= \frac{p^{2(r-2)}(p-1)(p^r+p^{r-1}-6)}{12},$$
as desired. The parity assertion follows directly from this expression.
\end{proof}

\begin{remark}
Note that  the last paragraph of the above argument can be used to give a different proof of the   $p\equiv 1\mod 4$ case of Theorem~\ref{thm:main}. Note  also  that using \eqref{temp-d2} one can determine the multiplicity of every irreducible invariant character of $\mathfrak{sl}_2(\mathbb{F}_p)$ in $S$.
\end{remark}

\bibliographystyle{alpha}
\bibliography{zchenrefs}

\end{document}